\documentclass[a4paper,11pt]{amsart}

\usepackage{amsmath, amsthm, amssymb}

\newcommand{\bC}{\mathbb{C}}
\newcommand{\sP}{\mathsf{P}}
\newcommand{\bP}{\mathbb{P}}
\newcommand{\rd}{\mathrm{d}}
\newcommand{\bN}{\mathbb{N}}
\newcommand{\cS}{\mathcal{S}}
\newcommand{\can}{\operatorname{can}}
\newcommand{\Id}{\mathrm{Id}}
\newcommand{\Fix}{\operatorname{Fix}}
\newcommand{\SAT}{\mathit{SAT}}
\newcommand{\sF}{\mathsf{F}}
\newcommand{\sJ}{\mathsf{J}}
\newcommand{\Res}{\operatorname{Res}}

\numberwithin{equation}{section}
\theoremstyle{plain}
\newtheorem{theorem}{Theorem}[section]
\newtheorem{lemma}[theorem]{Lemma}
\newtheorem{proposition}[theorem]{Proposition}
\newtheorem{mainth}{Theorem}

\theoremstyle{definition}

\newtheorem{notation}[theorem]{Notation}

\newtheorem*{acknowledgement}{Acknowledgement}
\theoremstyle{remark}
\newtheorem{remark}[theorem]{Remark}

\begin{document} 

\title[Quantitative approximations of the Lyapunov exponent]{Quantitative approximations of the Lyapunov exponent of a rational function over valued fields}

\author[Y\^usuke Okuyama]{Y\^usuke Okuyama}

\address{
Division of Mathematics,
Kyoto Institute of Technology,
Sakyo-ku, Kyoto 606-8585 Japan.}
\email{okuyama@kit.ac.jp}


\date{\today}

\begin{abstract}
We establish a quantitative approximation formula of the Lyapunov exponent of
a rational function of degree more than one over an algebraically closed field
of characteristic $0$ that is complete with respect to a non-trivial and possibly non-archimedean
absolute value, in terms of the multipliers of periodic points of the rational function. 
This quantifies both our former convergence result over general fields
and the one-dimensional version of Berteloot--Dupont--Molino's one
over archimedean fields.
\end{abstract}


\maketitle

\section{Introduction}\label{sec:intro}

In this article, we establish a quantitative logarithmic equidistribution result
for periodic points of a rational function over a more general field than 
that of the complex numbers,
using potential theory on the Berkovich projective line. 
Let $K$ be an algebraically closed field of characteristic $0$ that is complete with respect to
a non-trivial and possibly non-archimedean absolute value $|\cdot|$. 
We note that $K\cong\bC$ if and only if $K$ is archimedean.
The Berkovich projective line $\sP^1=\sP^1(K)$ over $K$ is
a compactification of the (classical) projective line $\bP^1=\bP^1(K)$ over $K$
and contains $\bP^1$ as a dense subset. We also note that $\sP^1\cong\bP^1$ if and only if
$K$ is archimedean.

Let $f\in K(z)$ be a rational function over $K$ of degree $>1$, and let $\mu_f$ be the 
equilibrium (or canonical) measure of $f$ on $\sP^1$.
The chordal derivative $f^{\#}$ of $f$ with respect to the normalized chordal metric on $\bP^1$ extends to a continuous function on $\sP^1$. 
The multiplier of a fixed point $w\in\bP^1$ of $f^n$ for some $n\in\bN$ 
is denoted by $(f^n)'(w)$.
For every fixed point $w\in\bP^1$ of $f$, we have $f^\#(w)=|f'(w)|$, 
and the function
$\log(f^\#)$ on $\sP^1$ has a logarithmic singularity at each critical point of $f$ in $\bP^1$.
The Lyapunov exponent of $f$ with respect to $\mu_f$ is defined by
\begin{gather}
 L(f):=\int_{\sP^1}\log(f^{\#})\rd\mu_f,\label{eq:lyapunov}
\end{gather}
which is in $(-\infty,\infty)$.

Our principal result is the following {\itshape quantitative} approximation of $L(f)$
by the $\log$ of the moduli of the multipliers of {\itshape non-superattracting} periodic points of $f$ in $\bP^1$, the qualitative version (i.e., with no non-trivial order estimates) of which
was obtained in \cite[Theorem 1]{OkuLog} (see also Szpiro--Tucker \cite{ST05}
for the qualitative version 
when $f$ is defined over a number field or a function field):
for each $n\in\bN$, let $\Fix(f^n)$ be the set of all fixed points 
of $f^n$ in $\bP^1$, and $\Fix^*(f^n)$ be the set of all periodic points $w$ of $f$
in $\bP^1$ having the {\itshape exact} period $n$ in that $w\in\Fix(f^n)\setminus(\bigcup_{m\in\bN:\, m|n\text{ and }m<n}\Fix(f^m))$. 

\begin{mainth}\label{th:lyapunov}
Let $f\in K(z)$ be a rational function of degree $d>1$
over an algebraically closed field $K$ of characteristic $0$ that is complete with respect to a non-trivial and possibly non-archimedean absolute value $|\cdot|$. 
Then
\begin{gather}
 L(f)=\frac{1}{nd^n}\sum_{w\in\Fix(f^n):\, (f^n)'(w)\neq 0}\log|(f^n)'(w)|+O(nd^{-n})
\quad\text{as }n\to\infty,\quad\text{and}\label{eq:quantitativeapprox}\\ 
 L(f)=\frac{1}{nd^n}
\sum_{w\in\Fix^*(f^n)}\log|(f^n)'(w)|+O(d^{-n/2})\quad\text{as }n\to\infty.\label{eq:exact} 
\end{gather}
\end{mainth}

In the case that $f$ has at most finitely many attracting 
periodic points in $\bP^1$, 
Theorem \ref{th:lyapunov} yields the following
quantitative approximations of $L(f)$ by the multipliers of 
{\itshape repelling} periodic points of $f$ in $\bP^1$: for each $n\in\bN$, set
$R(f^n):=\{w\in\Fix(f^n):|(f^n)'(w)|>1\}$
and $R^*(f^n):=\{w\in\Fix^*(f^n):|(f^n)'(w)|>1\}$.

\begin{mainth}\label{th:repelling}
Let $f\in K(z)$ be a rational function of degree $d>1$
over an algebraically closed field $K$ of characteristic $0$ that is complete with respect to a non-trivial and possibly non-archimedean absolute value $|\cdot|$. 

If $f$ has at most finitely many attracting 
periodic points in $\bP^1$, then
\begin{gather}
 L(f)=\frac{1}{nd^n}\sum_{w\in R(f^n)}\log|(f^n)'(w)|+O(nd^{-n})\quad\text{as }n\to\infty,\quad\text{and}\tag{\ref{eq:quantitativeapprox}'}\label{eq:repelling}\\
 L(f)=\frac{1}{nd^n}\sum_{w\in R^*(f^n)}\log|(f^n)'(w)|
+O(d^{-n/2})\quad\text{as }n\to\infty.
\tag{\ref{eq:exact}'}\label{eq:strict}
\end{gather}
\end{mainth}

For archimedean $K$, the finiteness assumption in Theorem \ref{th:repelling}
always holds (cf.\ \cite[Theorem 8.6]{Milnor3rd});
for non-archimedean $K$, any
periodic point of the polynomial $f(z)=z^d$ is (super) attracting if $|d|<1$.
The qualitative version (i.e., with no non-trivial order estimates) 
of Theorem \ref{th:repelling} for archimedean $K$ also follows from
Berteloot--Dupont--Molino \cite[Corollary 1.6]{BDM08};
see also Berteloot \cite{BertelootLyapunov}.
Their proofs are based on the positivity of $L(f)$, which is the case
for archimedean $K$ by Ruelle's inequality; 
for non-archimedean $K$, the polynomial $f(z)=z^d$ has
$L(f)=\log|d|\le 0$.
Our proofs of Theorems \ref{th:lyapunov} and \ref{th:repelling}
are independent of whether $L(f)$ is positive or not.

A bit surprisingly, the proofs of Theorems \ref{th:lyapunov} and \ref{th:repelling}
are independent of the equidistribution theorem for periodic points of $f$ in $\bP^1$
towards the equilibrium measure $\mu_f$, which was
due to Lyubich \cite[Theorem 3]{Lyubich83} for archimedean $K$ and
due to Favre--Rivera-Letelier \cite[Th\'eor\`eme B]{FR09} 
for non-archimedean $K$ of characteristic $0$. 

\subsection*{Organization of this article}
In Section \ref{sec:background}, we prepare a background on dynamics of rational functions
over general fields. 
In Section \ref{sec:quantitative}, we show Theorem \ref{th:lyapunov}:
let us remark that the proof can be simplified if there are no superattracting periodic points (see Lemma \ref{th:lower}). 
In Section \ref{sec:complex}, we show Theorem \ref{th:repelling} based on Theorem \ref{th:lyapunov}.

\section{Background}
\label{sec:background}
For the foundations of potential theory on $\sP^1$, see 
\cite[\S 5 and \S 8]{BR10}, \cite[\S 7]{FJbook}, 
\cite[\S 1-\S 4]{Jonsson12}, \cite[Chapter III]{Tsuji59}. 
For a potential theoretic study of dynamics on $\sP^1$,
see \cite[\S 10]{BR10}, \cite[\S 3]{FR09}, \cite[\S 5]{Jonsson12}, \cite[Chapitre VIII]{BM01}. 
See also \cite{Benedetto10}, \cite{Juan03} including non-archimedean dynamics.

\subsection*{Chordal metric on $\bP^1$}
Let $K$ be an algebraically closed field complete with respect to
a non-trivial and possibly non-archimedean absolute value $|\cdot|$. 
For a while, we allow $K$ to have any characteristic.
Let $\|(p_0,p_1)\|$ be the maximum norm 
$\max\{|p_0|,|p_1|\}$ 
on $K^2$ (for non-archimedean $K$) 
or the Euclidean norm $\sqrt{|p_0|^2+|p_1|^2}$ on $K^2$ (for archimedean $K$).
The origin of $K^2$ is also denoted by $0$, and 
$\pi$ is the canonical projection $K^2\setminus\{0\}\to\bP^1=\bP^1(K)$.
Setting the wedge product $(z_0,z_1)\wedge(w_0,w_1):=z_0w_1-z_1w_0$ on $K^2\times K^2$,
the normalized chordal metric $[z,w]$ on $\bP^1$ is the function
\begin{gather}
 (z,w)\mapsto [z,w]:=|p\wedge q|/(\|p\|\cdot\|q\|)(\le 1)\label{eq:chordaldist}
\end{gather}
on $\bP^1\times\bP^1$, where $p\in\pi^{-1}(z),q\in\pi^{-1}(w)$. Although
the topology of the Berkovich projective line $\sP^1=\sP^1(K)$,
which is a compactification of $\bP^1$, 
is not always metrizable, the relative topology of $\bP^1$ 
coincides with the metric topology on $\bP^1$ induced by the normalized
chordal metric. 

\subsection*{Hsia kernel on $\sP^1$}
Let $\delta_{\cS}$ be the Dirac measure on $\sP^1=\sP^1(K)$ at $\cS\in\sP^1$.
Let $\Omega_{\can}$ be the Dirac measure $\delta_{\cS_{\can}}$ at the canonical
(or Gauss) point $\mathcal{S}_{\can}\in\sP^1$ for non-archimedean
$K$ (\cite[\S1.2]{BR10}, \cite[\S2.1]{FR09}) or
the Fubini--Study area element $\omega$ on $\bP^1$ 
normalized as $\omega(\bP^1)=1$ for archimedean $K$. 
For non-archimedean $K$, the normalized chordal metric on $\bP^1$
canonically extends to the {\itshape generalized Hsia kernel $[\cS,\cS']_{\can}$
with respect to $\mathcal{S}_{\can}$ on} $\sP^1$ (for the construction,
see \cite[\S4.4]{BR10}, \cite[\S 2.1]{FR09}),
which vanishes if and only if $\cS=\cS'\in\bP^1$. 
For archimedean $K$, $[z,w]_{\can}$ is defined by $[z,w]$, by convention.
Let $\Delta$ be the Laplacian on $\sP^1$
(for the construction in the non-archimedean case,
see \cite[\S 5]{BR10}, \cite[\S7.7]{FJbook}, \cite[\S 3]{Thuillierthesis})
normalized so that for each $\cS\in\sP^1$, 
\begin{gather*}
 \Delta\log[\cdot,\cS]_{\can}=\delta_{\cS}-\Omega_{\can}\quad\text{on }\sP^1 
\end{gather*}
(for non-archimedean $K$, see \cite[Example 5.19]{BR10}, \cite[\S2.4]{FR09}: 
in \cite{BR10} the opposite sign convention on $\Delta$ is adopted).

\subsection*{Potential theory on $\sP^1$}
A {\itshape continuous weight $g$ on} $\sP^1$ is
a continuous function on $\sP^1$ such that $\mu^g:=\Delta g+\Omega_{\can}$
is a probability Radon measure on $\sP^1$. For a continuous weight $g$ on $\sP^1$,
the $g$-potential kernel 
\begin{gather*}
 \Phi_g(\cS,\cS'):=\log[\cS,\cS']_{\can}-g(\cS)-g(\cS')
\end{gather*}
(the negative of an Arakelov Green function of $\mu^g$) 
on $\sP^1$ is an upper semicontinuous function on $\sP^1\times\sP^1$.
The ($\exp$ of the) $\Phi_g$ is
separately continuous in each of the variables $\cS,\cS'\in\sP^1$, and
it introduces the $g$-potential 
\begin{gather*}
 U_{g,\nu}(\cdot):=\int_{\sP^1}\Phi_g(\cdot,\cS)\rd\nu(\cS)
\end{gather*}
on $\sP^1$ of each Radon measure $\nu$ on $\sP^1$. By the Fubini theorem,
$\Delta U_{g,\nu}=\nu-\nu(\sP^1)\mu^g$ on $\sP^1$. 
The $g$-equilibrium energy $V_g\in[-\infty,+\infty)$ of $\sP^1$ 
is the supremum of the energy functional
\begin{gather*}
 \nu\mapsto\int_{\sP^1\times\sP^1}\Phi_g\rd(\nu\times\nu)
\end{gather*}
on the space of all probability Radon measures on $\sP^1$. 
Indeed $V_g\in(-\infty,\infty)$ since 
$\int_{\sP^1\times\sP^1}\Phi_g\rd(\Omega_{\can}\times\Omega_{\can})
>-\infty$.
The variational characterization of $\mu^g$
asserts that the above energy functional attains the supremum uniquely at $\nu=\mu^g$.
Moreover,
\begin{gather*}
 U_{g,\mu^g}\equiv V_g\quad\text{on }\sP^1 
\end{gather*}
(for non-archimedean $K$, see \cite[Theorem 8.67 and Proposition 8.70]{BR10}).
A continuous weight $g$ on $\sP^1$ is a {\itshape normalized weight on} $\sP^1$ if
$V_g=0$. For a continuous weight $g$ on $\sP^1$,
$\overline{g}:=g+V_g/2$ is the unique normalized weight on $\sP^1$ 
satisfying $\mu^{\overline{g}}=\mu^g$.

\subsection*{Equilibrium measure $\mu_f$}
A rational function $f\in K(z)$ of degree $d>1$
extends to a continuous, surjective, open, and discrete 
endomorphism of $\sP^1$, preserving $\bP^1$ and $\sP^1\setminus\bP^1$, respectively,
and induces a push-forward $f_*$ and a pullback $f^*$ 
on the space of continuous functions on $\sP^1$ and, by duality,
on the space of
Radon measures on $\sP^1$ (\cite[\S9]{BR10}, \cite[\S2.2]{FR09}). 
A {\itshape non-degenerate homogeneous lift} 
$F=(F_0,F_1)$ of (the unextended) $f$ is a homogeneous polynomial endomorphism on $K^2$ 
such that $\pi\circ F=f\circ\pi$ on $K^2\setminus\{0\}$
and that $F^{-1}(0)=\{0\}$. 
The latter condition is equivalent to $\Res F\in K\setminus\{0\}$
(for the definition of the homogeneous resultant $\Res F=\Res(F_0,F_1)$ of $F$,
see, e.g., \cite[\S2.4]{SilvermanDynamics}).
Such $F$ is unique up to multiplication in $K\setminus\{0\}$, and
has the algebraic degree $d$, i.e., $\deg F_0=\deg F_1=\deg f=d$.
For every $n\in\bN$, $F^n$ is a non-degenerate homogeneous lift of $f^n$, 
and the function 
\begin{gather}
 T_{F^n}:=\log\|F^n\|-d^n\log\|\cdot\|\label{eq:descend}
\end{gather} 
on $K^2\setminus\{0\}$ descends to $\bP^1$ and in turn extends continuously to $\sP^1$,
satisfying $\Delta T_{F^n}=(f^n)^*\Omega_{\can}-d^n\Omega_{\can}$
on $\sP^1$ (see, e.g., \cite[Definition 2.8]{Okucharacterization}). 
The uniform limit 
$g_F:=\lim_{n\to\infty}T_{F^n}/d^n$ on $\sP^1$ exists and is indeed a continuous weight on $\sP^1$.
The {\itshape equilibrium $($or canonical$)$ measure of} $f$ is
the probability Radon measure
\begin{gather*}
 \mu_f:=\Delta g_F+\Omega_{\can}
=\lim_{n\to\infty}d^{-n}(f^n)^*\Omega_{\can}\quad\text{weakly on }\sP^1,
\end{gather*}
which is independent of the choice of the lift $F$ 
and satisfies that $f^*\mu_f=d\cdot\mu_f$
and $f_*\mu_f=\mu_f$ on $\sP^1$ (for non-archimedean $K$,
see \cite[\S10]{BR10}, \cite[\S2]{ChambertLoir06}, \cite[\S3.1]{FR09}).
The {\itshape dynamical Green function $g_f$ of $f$ on} $\sP^1$ is 
the unique normalized weight on $\sP^1$ such that $\mu^{g_f}=\mu_f$.
By the energy formula 
$V_{g_F}=-(\log|\Res F|)/(d(d-1))$ (due to DeMarco \cite{DeMarco03} 
for archimedean $K$
and due to Baker--Rumely \cite{BR06} 
when $f$ is defined over a number field;
see \cite[Appendix A]{Baker09} for a simple proof which works for general $K$)
and $\Res(cF)=c^{2d}\Res F$ for each $c\in K\setminus\{0\}$
(cf.\ \cite[Proposition 2.13(b)]{SilvermanDynamics}),
there is a non-degenerate homogeneous lift $F$ of $f$ satisfying
$V_{g_F}=0$, or equivalently, that $g_F=g_f$ on $\sP^1$.
We note that 
$U_{g_f,\mu_f}\equiv 0$ on $\sP^1$ 
and that for every $n\in\bN$, $\mu_{f^n}=\mu_f$ and $g_{f^n}=g_f$ on $\sP^1$.

\subsection*{Logarithmic proximity function $\Phi(f^n,\Id)_{g_f}$}
For more details on
the following, see \cite[Proposition 2.9]{Okucharacterization}.

\begin{proposition}
 For rational functions $\phi_i\in K(z)$ of degree $d_i$, $i\in\{1,2\}$,
 on $\bP^1$ satisfying $\phi_1\not\equiv\phi_2$ and $\max\{d_1,d_2\}>0$, 
 the function $z\mapsto[\phi_1(z),\phi_2(z)]$ on $\bP^1$
 extends continuously to {\itshape a function 
 $\cS\mapsto[\phi_1,\phi_2]_{\can}(\cS)$ on} $\sP^1$.
\end{proposition}

For each $n\in\bN$,
we introduce the {\itshape logarithmic proximity function weighted by $g_f$
\begin{gather*}
 \Phi(f^n,\Id)_{g_f}(\cdot):=\log[f^n,\Id]_{\can}(\cdot)-g_f\circ f^n(\cdot)-g_f(\cdot)
\end{gather*}
between $f^n$ and $\Id$ on} $\sP^1$, and set
\begin{gather}
 [f^n=\Id]:=\sum_{w\in\bP^1:f^n(w)=w}\delta_w\quad\text{on }\sP^1,\label{eq:roots} 
\end{gather}
where the sum takes into account the multiplicity of each
root $w\in\bP^1$ of $f^n=\Id$. 

For a proof of the following, see, e.g., \cite[Lemma 2.19]{Okucharacterization}.

\begin{lemma}[{cf. \cite[(1.4)]{Sodin92}}]
For every $n\in\bN$,
\begin{gather*}
 \Phi(f^n,\Id)_{g_f}(\cdot)
 =U_{g_f,[f^n=\Id]-(d^n+1)\mu_f}+\int_{\sP^1}\Phi(f^n,\Id)_{g_f}\rd\mu_f
\end{gather*}
on $\sP^1$. Since $U_{g_f,\mu_f}\equiv 0$ on $\sP^1$, this is rewritten as
\begin{gather}
 \Phi(f^n,\Id)_{g_f}(\cdot)
 =U_{g_f,[f^n=\Id]}+\int_{\sP^1}\Phi(f^n,\Id)_{g_f}\rd\mu_f\quad\text{on }\sP^1.\label{eq:Riesz}
\end{gather}
\end{lemma}

\subsection*{Chordal derivative $f^\#$}
The multiplier of a fixed point $w\in\bP^1$ of 
$f^n$ for some $n\in\bN$ is denoted by $(f^n)'(w)$.
A fixed point $w\in\bP^1$ of $f^n$ for some $n\in\bN$
is said to be superattracting, attracting, or repelling
if $(f^n)'(w)=0$, $|(f^n)'(w)|<1$, or $|(f^n)'(w)|>1$, respectively. 

In the rest of this subsection, we suppose that $K$ has characteristic $0$.
\begin{notation}
 Let $C(f)$ be the set of all critical points $c$ of $f$ in $\bP^1$,
 i.e., $f'(c)=0$.
\end{notation}
For every $n\in\bN$,
$f^n$ has $2d^n-2$ critical points in $\bP^1$ 
if we take into account the multiplicity of each $c\in C(f^n)$.
Let $\SAT(f)$ be the set of all superattracting periodic points of $f$ in $\bP^1$,
i.e., 
\begin{gather*}
 \SAT(f)=\bigcup_{n\in\bN}(\Fix(f^n)\cap C(f^n)). 
\end{gather*}
By $\#C(f)<\infty$ and the chain rule, $\#\SAT(f)<\infty$. 

The chordal derivative $f^\#$ on $\bP^1$ is a function
\begin{gather*}
\bP^1\ni z\mapsto f^\#(z):=\lim_{\bP^1\ni w\to z}[f(w),f(z)]/[w,z].
\end{gather*}
For every non-degenerate homogeneous lift $F$ of $f$,
there exists a sequence $(C_j^F)_{j=1}^{2d-2}$ in $K^2\setminus\{0\}$
such that the Jacobian determinant of $F$ factors as
\begin{gather*}
 \det DF(\cdot)=\prod_{j=1}^{2d-2}(\cdot\wedge C_j^F)\quad\text{on }K^2.
\end{gather*}
Setting $c_j:=\pi(C_j^F)$ $(j=1,2,\ldots,2d-2)$,
the sequence $(c_j)_{j=1}^{2d-2}$ in $\bP^1$
is independent of the choice of the lift $F$ upto permutation, and
satisfies that for every $c\in C(f)$, $\#\{j\in\{1,2,\ldots,2d-2\}:c_j=c\}-1$ equals
the multiplicity of $c$. 
For every $z\in\bP^1$, by a computation involving Euler's identity, we have
\begin{gather*}
 f^\#(z)=\frac{1}{|d|}\frac{\|p\|^2}{\|F(p)\|^2}|\det DF(p)|\quad\text{for }p\in\pi^{-1}(z)
\end{gather*}
(cf.\ \cite[Theorem 4.3]{Jonsson98}), which with \eqref{eq:chordaldist} yields
the equality $\log(f^{\#})=-\log|d|+\sum_{j=1}^{2d-2}(\log[\cdot,c_j]+\log\|C_j^F\|)-2T_F|\bP^1$
on $\bP^1$. The ($\exp$ of the) right hand side extends $f^\#$ to a continuous function
on $\sP^1$ so that
\begin{gather}
 \log(f^{\#})=-\log|d|+\sum_{j=1}^{2d-2}(\log[\cdot,c_j]_{\can}+\log\|C_j^F\|)-2T_F\quad\text{on }\sP^1,\label{eq:derivative}
\end{gather}
where the continuous extension of $f^\#$ is also denoted by the same $f^\#$.
The chain rule for $f^\#$ on $\bP^1$ extends to $\sP^1$.

For completeness, we include a proof of the following.

\begin{lemma}[{\cite[Lemma 3.6]{OkuFekete}}]\label{th:multiplier}
On $\sP^1$,
\begin{gather}
 \log(f^\#)=L(f)+\sum_{c\in C(f)}\Phi_{g_f}(\cdot,c)+2g_f\circ f-2g_f.\label{eq:formula}
\end{gather}
Here the sum over $C(f)$ takes into account the multiplicity 
of each $c\in C(f)$. 
\end{lemma}

\begin{proof}
Let us choose a non-degenerate homogeneous lift $F$ of $f$ so that 
$g_F=g_f$, i.e., $g_f=\lim_{n\to\infty}T_{F^n}/d^n$ on $\sP^1$. 

By the definition of $\Phi_{g_f}$, \eqref{eq:derivative} is rewritten as
\begin{gather}
 \log(f^{\#})=-\log|d|+\sum_{j=1}^{2d-2}(\Phi_{g_f}(\cdot,c_j)+g_f(c_j)+\log\|C_j^F\|)
 -2T_F+(2d-2)g_f\tag{\ref{eq:derivative}'}\label{eq:derivativedyn}
\end{gather} 
on $\sP^1$. We claim that $g_f\circ f=d\cdot g_f-T_F$ on $\sP^1$,
which is equivalent to $-2T_F+(2d-2)g_f=2g_f\circ f-2g_f$ on $\sP^1$;
Indeed, for every $z\in\bP^1$, by $g_f=g_F=\lim_{n\to\infty}T_{F^n}/d^n$ on $\sP^1$,
we have
\begin{multline*}
 g_f\circ f(z)-d\cdot g_f(z)
=\lim_{n\to\infty}\left(\frac{1}{d^n}(\log\|F^n(F(p))\|-d^n\log\|F(p)\|)\right)\\
-d\cdot\lim_{n\to\infty}\left(\frac{1}{d^{n+1}}(\log\|F^{n+1}(p)\|-d^{n+1}\log\|p\|)\right)\\
=-(\log\|F(p)\|-d\cdot\log\|p\|)=-T_F(z),
\end{multline*}
where $p\in\pi^{-1}(z)$. Hence $g_f\circ f-d\cdot g_f=-T_F$ on $\bP^1$,
which in turn holds on $\sP^1$ by the continuity of both sides, and the claim holds.

By this claim, \eqref{eq:derivativedyn} is rewritten as
\begin{multline}
 \log(f^{\#})=-\log|d|+\sum_{j=1}^{2d-2}(\Phi_{g_f}(\cdot,c_j)+g_f(c_j)+\log\|C_j^F\|)
 +2g_f\circ f-2g_f\quad\text{on }\sP^1.\label{unintegrated}
\end{multline}
Integrating both sides in \eqref{unintegrated} against $\rd\mu_f$ over $\sP^1$, 
by $U_{g_f,\mu_f}\equiv 0$ and $f_*\mu_f=\mu_f$ on $\sP^1$, we have
\begin{align*}
 &L(f):=\int_{\sP^1}\log(f^{\#})\rd\mu_f\\
=&-\log|d|+\sum_{j=1}^{2d-2}(U_{g_f,\mu_f}(c_j)+g_f(c_j)+\log\|C_j^F\|)
+2\int_{\sP^1}g_f\circ f\rd\mu_f-2\int_{\sP^1}g_f\rd\mu_f\\
=&-\log|d|+\sum_{j=1}^{2d-2}(g_f(c_j)+\log\|C_j^F\|).
\end{align*}
This with \eqref{unintegrated} completes the proof of Lemma \ref{th:multiplier}.
\end{proof}

\subsection*{Berkovich Julia and Fatou sets $\sJ(f)$ and $\sF(f)$}
The exceptional set of (the extended) $f$ is 
$E(f):=\left\{a\in\bP^1:\#\bigcup_{n\in\bN}f^{-n}(a)<\infty\right\}$,
which agrees with the set of all $a\in\SAT(f)$ such that $\deg_{f^j(a)} f=d$ for any $j\in\bN$.
The {\itshape Berkovich} Julia set of $f$ is
\begin{gather*}
 \sJ(f):=\left\{\cS\in\sP^1:
 \bigcap_{V:\text{ open in }\sP^1\text{ and contains } \cS}\left(\bigcup_{n\in\bN}f^n(V)\right)=\sP^1\setminus E(f)\right\}
\end{gather*}
(cf.\ \cite[Definition 2.8]{FR09}), which is closed in $\sP^1$, and
the {\itshape Berkovich} Fatou set of $f$ is 
$\sF(f):=\sP^1\setminus\sJ(f)$, which is open in $\sP^1$.
For archimedean $K$, these definitions
of $\sJ(f)$ and $\sF(f)$ are equivalent to those 
of the Julia and Fatou sets of $f$
in terms of the non-normality and the normality of $\{f^n:n\in\bN\}$, respectively.

A Berkovich Fatou component of $f$ is a connected component of $\sF(f)$;
if $W$ is a Berkovich Fatou component of $f$, then
so is $f(W)$.
A Berkovich Fatou component $W$ of $f$ is cyclic under $f$ if $f^p(W)=W$ for some $p\in\bN$. 
For archimedean $K$, the classification of cyclic (Berkovich) Fatou components into 
\begin{itemize}
 \item immediate attractive basins of either (super)attracting or parabolic cycles, and
 \item rotation domains, i.e., Siegel disks and Herman rings
\end{itemize}
is essentially due to Fatou (cf.\ \cite[Theorem 5.2]{Milnor3rd}).
For non-archimedean $K$, 
its counterpart due to Rivera-Letelier (see \cite[Proposition 2.16]{FR09} and 
its {\itshape esquisse de d\'emonstration} and also \cite[Remark 7.10]{Benedetto10})
asserts that {\itshape every cyclic Berkovich Fatou component 
$W$ of $f$ is either 
\begin{itemize}
 \item  an immediate attractive basin of $f$ in that
 $W$ contains a $($super$)$attracting fixed point $a$ of $f^p$ in $W\cap\bP^1$ 
 for some $p\in\bN$ and that $\lim_{n\to\infty}(f^p)^n(w)=a$ for any $w\in W$, or
 \item a singular domain of $f$ in that $f^p(W)=W$ and
 that $f^p:W\to W$ is injective for some $p\in\bN$,
\end{itemize}
and moreover, only one of these two possibilities occurs.}

\section{Proof of Theorem \ref{th:lyapunov}}
\label{sec:quantitative}
Let $K$ be an algebraically closed field of characteristic $0$ that is complete with respect to
a non-trivial and possibly non-archimedean absolute value $|\cdot|$. Let $f\in K(z)$ be
a rational function over $K$ of degree $d>1$. 

We note that $[f^n=\Id]/(d^n+1)$ is a probability Radon 
measure on $\sP^1$, and that
$\sup_{n\in\bN}([f^n=\Id](\SAT(f)))\le\#\SAT(f)<\infty$ 
for every $n\in\bN$.

The strategy of the proof of Theorem \ref{th:lyapunov}
is to compute the difference
$\int_{\sP^1\setminus\SAT(f)}\log(f^\#)\rd[f^n=\Id]
-([f^n=\Id](\sP^1\setminus\SAT(f))) L(f)$
in two different ways, by integrating the equality \eqref{eq:formula} applied to $f^n$ 
and the \eqref{eq:formula} itself
against $[f^n=\Id]$ over $\sP^1\setminus\SAT(f)$, for each $n\in\bN$.


\begin{lemma}\label{th:lower}
For every $n\in\bN$,
\begin{align}
\notag &\frac{1}{d^n}\int_{\sP^1\setminus\SAT(f)}\log(f^\#)\rd[f^n=\Id]
-\frac{[f^n=\Id](\sP^1\setminus\SAT(f))}{d^n}L(f)\\
=&\frac{1}{nd^n}
\sum_{c\in(C(f^n)\setminus\Fix(f^n))
\cap(\bigcup_{j=0}^{n-1}f^{-j}(C(f)\setminus\SAT(f)))}\Phi_{g_f}(f^n(c),c)\label{eq:first}\\
&+\frac{1}{nd^n}
\sum_{c\in(C(f^n)\setminus\Fix(f^n))
\cap(\bigcup_{j=0}^{n-1}f^{-j}(C(f)\cap\SAT(f)))}
\Phi_{g_f}(f^n(c),c)\label{eq:second}\\
&-\frac{1}{nd^n}\sum_{c\in C(f^n)\setminus\Fix(f^n)}
 \int_{\SAT(f)}\Phi_{g_f}(c,\cdot)\rd[f^n=\Id](\cdot)\label{eq:third}\\
&-\frac{1}{nd^n}\sum_{c\in C(f^n)\cap\Fix(f^n)}\int_{\SAT(f)\setminus\{c\}}\Phi_{g_f}(c,\cdot)\rd[f^n=\Id](\cdot)\label{eq:fourth}\\
\notag &-\frac{2-2d^{-n}}{n}\int_{\sP^1}\Phi(f^n,\Id)_{g_f}\rd\mu_f.
\end{align}
Here the sums over subsets in $C(f^n)$
take into account the multiplicity of each $c$ as a critical point of $f^n$.
\end{lemma}

\begin{proof}
Integrating both sides in \eqref{eq:formula} against $\rd[f=\Id]$
over $\bP^1\setminus\SAT(f)$,
since $f_*[f=\Id]=[f=\Id]$ on $\sP^1\setminus\SAT(f)$, 
we have
\begin{multline}
\int_{\sP^1\setminus\SAT(f)}\log(f)^\#\rd[f=\Id]
-([f=\Id](\sP^1\setminus\SAT(f)))\cdot L(f)\\
=\sum_{c\in C(f)}U_{g_f,[f=\Id]|(\sP^1\setminus\SAT(f))}(c).\label{eq:approxiterated}
\end{multline}
We claim that for every $c\in C(f)$, 
\begin{multline}
U_{g_f,[f=\Id]|(\sP^1\setminus\SAT(f))}(c)
+\int_{\sP^1}\Phi(f,\Id)_{g_f}\rd\mu_f\\
=
\begin{cases}
 \Phi_{g_f}(f(c),c)
-\int_{\SAT(f)}\Phi_{g_f}(c,\cdot)\rd[f=\Id](\cdot) &
 \text{if }c\not\in\Fix(f),\\
\hspace*{30pt}-\int_{\SAT(f)\setminus\{c\}}\Phi_{g_f}(c,\cdot)\rd[f=\Id](\cdot)
& \text{if }c\in\Fix(f);
\end{cases}\label{eq:potentialresidual}
\end{multline}
Indeed, 
using \eqref{eq:Riesz} and $\Phi(f,\Id)_{g_f}=\Phi_{g_f}(f,\Id)$ on $\bP^1$, 
we have
\begin{multline*}
U_{g_f,[f=\Id]|(\sP^1\setminus\SAT(f))}(c)
=\lim_{\bP^1\ni z\to c}U_{g_f,[f=\Id]|(\sP^1\setminus\SAT(f))}(z)\\
=\lim_{\bP^1\ni z\to c}\left(\Phi_{g_f}(f(z),z)-\int_{\SAT(f)}\Phi_{g_f}(z,\cdot)\rd[f=\Id](\cdot)\right)
-\int_{\sP^1}\Phi(f,\Id)_{g_f}\rd\mu_f,
\end{multline*}
and moreover,
\begin{multline*}
\lim_{\bP^1\ni z\to c}\left(\Phi_{g_f}(f(z),z)-\int_{\SAT(f)}\Phi_{g_f}(z,\cdot)\rd[f=\Id](\cdot)\right)\\
=
\begin{cases}
\Phi_{g_f}(f(c),c)-\int_{\SAT(f)}\Phi_{g_f}(c,\cdot)\rd[f=\Id](\cdot) &
 \text{if }c\not\in\Fix(f),\\
\lim_{\bP^1\ni z\to c}(\Phi_{g_f}(f(z),z)-\Phi_{g_f}(z,c)) & \\
\hspace*{30pt}-\int_{\SAT(f)\setminus\{c\}}\Phi_{g_f}(c,\cdot)\rd[f=\Id](\cdot) & \text{if }c\in\Fix(f).
\end{cases}
\end{multline*}
In the latter case that $c\in C(f)\cap\Fix(f)$, the first term 
in the right hand side is computed as
\begin{multline*}
 \lim_{\bP^1\ni z\to c}(\Phi_{g_f}(f(z),z)-\Phi_{g_f}(z,c))=\lim_{\bP^1\ni z\to c}\log\frac{[f(z),z]}{[z,c]}\\
=\lim_{\bP^1\ni z\to c}\log\frac{|f(z)-f(c)+c-z|}{|z-c|}=\log|f'(c)-1|=\log|0-1|=0,
\end{multline*}
where we can assume $c\neq\infty$ by the coordinate change $w\mapsto 1/w$ when $c=\infty$. Hence the claim holds. 

For every $n\in\bN$,
from \eqref{eq:approxiterated} and \eqref{eq:potentialresidual}
applied to $f^n$,
we have
\begin{multline*}
\int_{\sP^1\setminus\SAT(f)}\log(f^n)^\#\rd[f^n=\Id]
-([f^n=\Id](\sP^1\setminus\SAT(f)))\cdot L(f^n)\\
=\sum_{c\in C(f^n)\setminus\Fix(f^n)}\Phi_{g_f}(f^n(c),c)
-\sum_{c\in C(f^n)\setminus\Fix(f^n)}
\int_{\SAT(f)}\Phi_{g_f}(c,\cdot)\rd[f^n=\Id](\cdot)\\
-\sum_{c\in C(f^n)\cap\Fix(f^n)}\int_{\SAT(f)\setminus\{c\}}\Phi_{g_f}(c,\cdot)\rd[f^n=\Id](\cdot)-(2d^n-2)\int_{\sP^1}\Phi(f^n,\Id)_{g_f}\rd\mu_f,
\end{multline*}
where $g_{f^n}=g_f$, $\mu_{f^n}=\mu_f$, and $\SAT(f^n)=\SAT(f)\cap\Fix(f^n)$,
and by the chain rule (and $f_*[f^n=\Id]=[f^n=\Id]$ on $\sP^1\setminus\SAT(f)$ and
$f_*\mu_f=\mu_f$ on $\sP^1$, recalling 
also the definition \eqref{eq:lyapunov} of $L(f)$), we have
\begin{multline*}
\int_{\sP^1\setminus\SAT(f)}\log(f^n)^\#\rd[f^n=\Id]
-([f^n=\Id](\sP^1\setminus\SAT(f)))\cdot L(f^n)\\
=n\left(
\int_{\sP^1\setminus\SAT(f)}\log(f^\#)\rd[f^n=\Id]
-([f^n=\Id](\sP^1\setminus\SAT(f)))\cdot L(f)\right).
\end{multline*}
Now the proof of Lemma \ref{th:lower} is complete
since $C(f^n)=\bigcup_{j=0}^{n-1}f^{-j}(C(f))$
and $C(f)=(C(f)\setminus\SAT(f))\cup(C(f)\cap\SAT(f))$.
\end{proof}

Let us estimate the terms \eqref{eq:first}, \eqref{eq:second}, \eqref{eq:third}, 
and \eqref{eq:fourth}.

\begin{lemma}\label{th:recurrence}
\begin{gather*}
\frac{1}{nd^n}
\sum_{c\in(C(f^n)\setminus\Fix(f^n))
\cap(\bigcup_{j=0}^{n-1}f^{-j}(C(f)\setminus\SAT(f)))}\Phi_{g_f}(f^n(c),c)
=O(1)\quad\text{as }n\to\infty.
\end{gather*} 
Here, for every $n\in\bN$, the sum takes into account the multiplicity of each $c$
as a critical point of $f^n$.
\end{lemma}

\begin{proof}
For every $n\in\bN$, by the definition of $\Phi_{g_f}$,
we have
\begin{align*}
&\frac{2d^n-2}{nd^n}\cdot 2\sup_{\sP^1}|g_f|\\
\ge&
\frac{1}{nd^n}
\sum_{c\in(C(f^n)\setminus\Fix(f^n))
\cap(\bigcup_{j=0}^{n-1}f^{-j}(C(f)\setminus\SAT(f)))}\Phi_{g_f}(f^n(c),c)\\
\ge&
\frac{1}{nd^n}
\sum_{c'\in C(f)\setminus\SAT(f)}\sum_{j=0}^{n-1}
\sum_{w\in f^{-j}(c')}
\log[f^n(w),w]-\frac{2d^n-2}{nd^n}\cdot 2\sup_{\sP^1}|g_f|,
\end{align*}
where the sums take into account the appropriate multiplicities
of $c,c',$ and $w$.

We can fix $L>1$ such that $f:\bP^1\to\bP^1$ is $L$-Lipschitz continuous
with respect to the normalized chordal metric
(for non-archimedean $K$, see, e.g., \cite[Theorem 2.14]{SilvermanDynamics}). 
Then for every $c'\in C(f)\setminus\SAT(f)$, 
every $j\in\{0,1,2,\ldots,n-1\}$, and every $w\in f^{-j}(c')$,
$L^n[f^n(w),w]\ge L^j[f^n(w),w]\ge[f^j(f^n(w)),f^j(w)]=[f^n(c'),c']$,
so that
\begin{gather*}
 \log[f^n(w),w]\ge\log[f^n(c'),c']-n\log L.
\end{gather*}
Recall the definition of the Berkovich Julia and Fatou sets $\sJ(f)$ and $\sF(f)$ 
in Section \ref{sec:background}.
We claim that for every $c'\in C(f)\setminus\SAT(f)$,
\begin{gather}
(0\ge)\log[f^n(c'),c']
\ge\begin{cases}
  O(1) & \text{if }c'\in\sF(f)\\
  O(n) & \text{if }c'\in\sJ(f)
 \end{cases}
\quad\text{as }n\to\infty;\label{eq:recurrence}
\end{gather}
indeed, in the former case,
if $\liminf_{n\to\infty}[f^n(c'),c']=0$ for some $c'\in (C(f)\cap\sF(f))\setminus\SAT(f)$,
then the Berkovich Fatou component $U$ of $f$
containing $c'$ is cyclic under $f$, i.e.,
$f^p(U)=U$ for some $p\in\bN$. Since $c'\in C(f)\cap U$, $f^p:U\to U$ is not injective, 
and by the classification of cyclic Berkovich Fatou components of $f$ 
(see Section \ref{sec:background}),
$U$ is an immediate attracting basin of an either (super)attracting
or parabolic cycle of $f$ in $\bP^1$. Then since $c'\not\in\SAT(f)$, 
$\liminf_{n\to\infty}[f^n(c'),c']>0$, which is a contradiction. On the other hand,
in the latter case, by (the proof of) Przytycki's lemma \cite[Lemma 1]{Przytycki93},
it holds that 
{\itshape for every $c'\in C(f)\cap\sJ(f)$ and every $n\in\bN$, $[f^n(c'),c']\ge 1/(20L^n)$}.

Hence the claim holds. 
Now we have
\begin{multline*}
 \frac{1}{nd^n}
\sum_{c'\in C(f)\setminus\SAT(f)}\sum_{j=0}^{n-1}
\sum_{w\in f^{-j}(c')}
\log[f^n(w),w]\\
\ge\frac{1}{nd^n}\sum_{c'\in C(f)\setminus\SAT(f)}\sum_{j=0}^{n-1}d^j(\log[f^n(c'),c']-n\log L)\\
\ge\frac{2d-2}{nd^n}\cdot O(n)\cdot\sum_{j=0}^{n-1}d^j=O(1)\quad\text{as }n\to\infty,
\end{multline*}
and the proof of Lemma \ref{th:recurrence} is complete.
\end{proof}

The following technical and elementary
lemma is useful.

\begin{lemma}\label{th:distance}
There exists $\delta>0$ such that
for every $a\in\SAT(f)$ and every $c\in C(f)$ 
satisfying $\bigcup_{j\in\bN\cup\{0\}}f^{-j}(c)\neq\{a\}$,
\begin{gather}
 \inf\left\{[a,w]:w\in\left(\bigcup_{j\in\bN\cup\{0\}}f^{-j}(c)\right)\setminus\{a\}\right\}\ge\delta.\label{eq:distance} 
\end{gather}
Moreover, if $a\in\SAT(f)$ and $c\in C(f)$ 
satisfy $\bigcup_{j\in\bN\cup\{0\}}f^{-j}(c)=\{a\}$, then 
$a\in\Fix(f)$. In particular, $\bigcup_{j\in\bN\cup\{0\}}f^{-j}(c)\subset\Fix(f)$.
\end{lemma}

\begin{proof}
 Let $a\in\SAT(f)$. Then there is $p\in\bN$ such that $f^p(a)=a$ and
 $f^j(a)\neq a$ for every $j\in\{1,2,\ldots,p-1\}$. 
 We can fix an open neighborhood $U$ of $a$ in $\bP^1$ so small that
 $f^p(U)\subset U$ by the Taylor expansion of $f^p$ at $a$ and that
 $f^\ell(U)$ $(\ell\in\{0,1,2,\ldots,p-1\})$
 are mutually disjoint. Set $\mathcal{O}_a:=\{f^\ell(a):\ell\in\{0,1,2,\ldots,p-1\}\}$ and
 $\mathcal{U}:=\bigcup_{\ell=0}^{p-1}f^\ell(U)$, so that $f(\mathcal{U})\subset\mathcal{U}$,
 $\mathcal{U}\cap U=U$ and $\mathcal{O}_a\cap U=\{a\}$.
 Decreasing $U$ if necessary, we can
 assume $f^{-1}(\mathcal{O}_a)\cap\mathcal{U}\subset\mathcal{O}_a$ and 
 $\mathcal{U}\cap C(f)\subset\mathcal{O}_a$ by 
 $\# f^{-1}(\mathcal{O}_a)<\infty$ and $\# C(f)<\infty$, respectively.

 Let $c\in C(f)$ satisfy $\bigcup_{j\in\bN\cup\{0\}}f^{-j}(c)\neq\{a\}$.
 If $(\bigcup_{j\in\bN\cup\{0\}}f^{-j}(c))\cap(U\setminus\{a\})\neq\emptyset$, 
 then $c\in\mathcal{U}\cap C(f)$ by $f(\mathcal{U})\subset\mathcal{U}$, 
 and $c\in\mathcal{O}_a$ by $\mathcal{U}\cap C(f)\subset\mathcal{O}_a$.
 Hence by $f^{-1}(\mathcal{O}_a)\cap\mathcal{U}\subset\mathcal{O}_a$, we have
 $(\bigcup_{j\in\bN\cup\{0\}}f^{-j}(c))\cap\mathcal{U}\subset\mathcal{O}_a$, 
 so $(\bigcup_{j\in\bN\cup\{0\}}f^{-j}(c))\cap U\subset\mathcal{O}_a\cap U=\{a\}$, 
 which is a contradiction. Hence 
 $\inf\{[a,w]:w\in(\bigcup_{j\in\bN\cup\{0\}}f^{-j}(c))\setminus\{a\}\}>0$,
 which with $\#\SAT(f)<\infty$ and $\#C(f)<\infty$ completes the proof of the former assertion. The latter assertion is obvious.
\end{proof}

\begin{lemma}\label{th:preperiodic}
\begin{multline*}
\sup_{n\in\bN}\left(\sup_{c\in(\bigcup_{j=0}^{n-1}f^{-j}(C(f)\cap\SAT(f)))\setminus\Fix(f^n)}
|\Phi_{g_f}(f^n(c),c)|\right)
\le-\log\delta+2\sup_{\sP^1}|g_f|<\infty. 
\end{multline*}
\end{lemma}

\begin{proof}
For every $n\in\bN$ and 
every $c\in(\bigcup_{j=0}^{n-1}f^{-j}(C(f)\cap\SAT(f)))\setminus\Fix(f^n)$,
we have $f^n(c)\in\SAT(f)$
and $c\in\bigcup_{j=0}^{n-1}f^{-j}(c')$ for some $c'\in C(f)$.
Since $c\not\in\Fix(f^n)$, by the latter assertion of Lemma \ref{th:distance},
$\bigcup_{j\in\bN\cup\{0\}}f^{-j}(c')\neq\{f^n(c)\}$. Hence
by \eqref{eq:distance}, we have either $c=f^n(c)$ or $[f^n(c),c]\ge\delta$,
but the former possibility does not occur since $c\not\in\Fix(f^n)$.

Hence $\inf_{n\in\bN}(\inf_{c\in(\bigcup_{j=0}^{n-1}f^{-j}(C(f)\cap\SAT(f)))\setminus\Fix(f^n)}[f^n(c),c])\ge\delta$.
Now the proof is complete by the definition of $\Phi_{g_f}$.
\end{proof}

\begin{lemma}\label{th:Bottcher}
\begin{multline*}
 \sup_{n\in\bN}
\left(\sup_{c\in C(f^n)\setminus\Fix(f^n)}
\left|\int_{\SAT(f)}\Phi_{g_f}(c,\cdot)\rd[f^n=\Id](\cdot)\right|\right)\\
\le\left(-\log\delta+2\sup_{\sP^1}|g_f|\right)
\left(\sup_{n\in\bN}([f^n=\Id](\SAT(f)))\right)<\infty.
\end{multline*}
\end{lemma}

\begin{proof}
For every $n\in\bN$, every $c\in C(f^n)\setminus\Fix(f^n)$,
and every $w\in\SAT(f)\cap\Fix(f^n)$,
we have $c\in\bigcup_{j=0}^{n-1}f^{-j}(c')$ for some $c'\in C(f)$.
Since $c\not\in\Fix(f^n)$, by the latter assertion of Lemma \ref{th:distance},
$\bigcup_{j\in\bN\cup\{0\}}f^{-j}(c')\neq\{w\}$.
Hence by \eqref{eq:distance}, we have either $c=w\in\Fix(f^n)$
or $[c,w]\ge\delta$,
but the former possibility does not occur since $c\not\in\Fix(f^n)$.

Hence $\inf_{n\in\bN}
(\inf_{c\in C(f^n)\setminus\Fix(f^n),w\in\SAT(f)\cap\Fix(f^n)}
[c,w])\ge\delta$.
Now the proof is complete by the definition of $\Phi_{g_f}$ and
$\sup_{n\in\bN}([f^n=\Id](\SAT(f)))<\infty$.
\end{proof}

\begin{lemma}\label{th:supreatt}
\begin{multline*}
\sup_{n\in\bN}\left(\sup_{c\in C(f^n)\cap\Fix(f^n)}
\left|\int_{\SAT(f)\setminus\{c\}}\Phi_{g_f}(c,\cdot)\rd[f^n=\Id](\cdot)\right|\right)\\
\le\left(-\log
\left(\inf_{c,c'\in\SAT(f):\, c\neq c'}[c,c']\right)+2\sup_{\sP^1}|g_f|\right)
\left(\sup_{n\in\bN}([f^n=\Id](\SAT(f)))\right)<\infty.
\end{multline*}
\end{lemma}

\begin{proof}
Since $\#\SAT(f)<\infty$, 
we have $\inf_{c,c'\in\SAT(f):\, c\neq c'}[c,c']\in(0,1)$.

Now the proof is complete by 
$\bigcup_{n\in\bN}(C(f^n)\cap\Fix(f^n))=\SAT(f)$,
the definition of $\Phi_{g_f}$, and
$\sup_{n\in\bN}([f^n=\Id](\SAT(f)))<\infty$.
\end{proof}

By Lemmas \ref{th:lower}, \ref{th:recurrence}, \ref{th:preperiodic}, 
\ref{th:Bottcher}, \ref{th:supreatt},
we have
\begin{multline}
\frac{1}{d^n}\int_{\sP^1\setminus\SAT(f)}\log(f^\#)\rd[f^n=\Id]
-\frac{[f^n=\Id](\sP^1\setminus\SAT(f))}{d^n}L(f)\\
=O(1)+3\cdot\frac{2d^n-2}{nd^n}\cdot O(1)
-\frac{2-2d^{-n}}{n}\int_{\sP^1}\Phi(f^n,\Id)_{g_f}\rd\mu_f\\
=-\frac{2-2d^{-n}}{n}\int_{\sP^1}\Phi(f^n,\Id)_{g_f}\rd\mu_f+O(1)
\quad\text{as }n\to\infty.\label{eq:iterates}
\end{multline}

By an argument similar to (and simpler than) the above, the following also holds.
\begin{lemma}\label{th:simple}
\begin{multline}
\frac{1}{d^n}\int_{\sP^1\setminus\SAT(f)}\log(f^\#)\rd[f^n=\Id]
-\frac{[f^n=\Id](\sP^1\setminus\SAT(f))}{d^n}L(f)\\
=-\frac{2d-2}{d^n}\int_{\sP^1}\Phi(f^n,\Id)_{g_f}\rd\mu_f+O(nd^{-n})\quad\text{as }n\to\infty.\label{eq:simple}
\end{multline}
\end{lemma}

\begin{proof}
In the following, 
the sums over subsets of $C(f)$
take into account the multiplicity of each $c$ as a critical point of $f$.

For every $n\in\bN$, integrating both sides in \eqref{eq:formula} itself
against $\rd[f^n=\Id]$ over $\bP^1\setminus\SAT(f)$,
since $f_*[f^n=\Id]=[f^n=\Id]$ on $\sP^1\setminus\SAT(f)$, we have
 \begin{multline*}
 \int_{\sP^1\setminus\SAT(f)}\log(f^\#)\rd[f^n=\Id]
 -([f^n=\Id](\sP^1\setminus\SAT(f)))\cdot L(f)\\
 =\sum_{c\in C(f)}U_{g_f,[f^n=\Id]|(\sP^1\setminus\SAT(f))}(c),
 \end{multline*}
and by \eqref{eq:potentialresidual} applied to $f^n$,
for each $c\in C(f)(\subset C(f^n))$,
\begin{multline*}
U_{g_f,[f^n=\Id]|(\sP^1\setminus\SAT(f))}(c)
+\int_{\sP^1}\Phi(f^n,\Id)_{g_f}\rd\mu_f\\
=
\begin{cases}
 \Phi_{g_f}(f^n(c),c)
-\int_{\SAT(f)}\Phi_{g_f}(c,\cdot)\rd[f^n=\Id](\cdot) &
 \text{if }c\not\in\Fix(f^n),\\
\hspace*{30pt}-\int_{\SAT(f)\setminus\{c\}}\Phi_{g_f}(c,\cdot)\rd[f^n=\Id](\cdot)
& \text{if }c\in\Fix(f^n),
\end{cases}
\end{multline*}
where $g_{f^n}=g_f$, $\mu_{f^n}=\mu_f$, and $\SAT(f^n)=\SAT(f)\cap\Fix(f^n)$.
Hence for every $n\in\bN$,
\begin{multline*}
\sum_{c\in C(f)}U_{g_f,[f^n=\Id]|(\sP^1\setminus\SAT(f))}(c)
+(2d-2)\int_{\sP^1}\Phi(f^n,\Id)_{g_f}\rd\mu_f\\
=\sum_{c\in C(f)\setminus\Fix(f^n)}\Phi_{g_f}(f^n(c),c)
-\sum_{c\in C(f)\setminus\Fix(f^n)}
 \int_{\SAT(f)}\Phi_{g_f}(c,\cdot)\rd[f^n=\Id](\cdot)\\
-\sum_{c\in C(f)\cap\Fix(f^n)}\int_{\SAT(f)\setminus\{c\}}\Phi_{g_f}(c,\cdot)\rd[f^n=\Id](\cdot) 
\end{multline*}
so, by $(C(f)\setminus\SAT(f))\setminus\Fix(f^n)=C(f)\setminus\SAT(f)$, 
we have
\begin{align}
 \notag
&\frac{1}{d^n}\int_{\sP^1\setminus\SAT(f)}\log(f^\#)\rd[f^n=\Id]
 -\frac{[f^n=\Id](\sP^1\setminus\SAT(f))}{d^n}L(f)\\
\tag{\ref{eq:first}'} 
=&\frac{1}{d^n}\sum_{c\in C(f)\setminus\SAT(f)}\Phi_{g_f}(f^n(c),c)\label{eq:firstsimple}\\
\tag{\ref{eq:second}'} 
&+\frac{1}{d^n}\sum_{c\in(C(f)\cap\SAT(f))\setminus\Fix(f^n)}\Phi_{g_f}(f^n(c),c)\label{eq:secondsimple}\\
\tag{\ref{eq:third}'} 
&-\frac{1}{d^n}\sum_{c\in C(f)\setminus\Fix(f^n)}\int_{\SAT(f)}\Phi_{g_f}(c,\cdot)\rd[f^n=\Id](\cdot)\label{eq:thirdsimple}\\
\tag{\ref{eq:fourth}'}
&-\frac{1}{d^n}\sum_{c\in C(f)\cap\Fix(f^n)}\int_{\SAT(f)\setminus\{c\}}\Phi_{g_f}(c,\cdot)\rd[f^n=\Id](\cdot)\label{eq:fourthsimple}\\
\notag
&-\frac{2d-2}{d^n}\int_{\sP^1}\Phi(f^n,\Id)_{g_f}\rd\mu_f.
\end{align}
By \eqref{eq:recurrence} and the definition of $\Phi_{g_f}$,
the term \eqref{eq:firstsimple} is estimated as
 \begin{align*}
 \frac{2d-2}{d^n}\cdot 2\sup_{\sP^1}|g_f|\ge
&\frac{1}{d^n}\sum_{c\in C(f)\setminus\SAT(f)}\Phi_{g_f}(f^n(c),c)\\
 \ge&\frac{1}{d^n}\sum_{c\in C(f)\setminus\SAT(f)}\log[f^n(c),c]
 -\frac{2d-2}{d^n}\cdot 2\sup_{\sP^1}|g_f|\\
 \ge&\frac{2d-2}{d^n}\cdot O(n)+O(d^{-n})=O(nd^{-n})\quad\text{as }n\to\infty,
\end{align*}
and since $C(f)\subset C(f^n)$,
by Lemmas \ref{th:preperiodic}, \ref{th:Bottcher}, and \ref{th:supreatt}
(or by arguments similar to and simpler than those in their proofs),
the terms \eqref{eq:secondsimple}, \eqref{eq:thirdsimple},
and \eqref{eq:fourthsimple} have 
the order $((2d-2)/d^n)\cdot O(1)=O(d^{-n})$ as $n\to\infty$.

Now the proof is complete.
\end{proof}

By \eqref{eq:iterates} and \eqref{eq:simple}, we have
\begin{multline*}
-\frac{2-2d^{-n}}{n}\int_{\sP^1}\Phi(f^n,\Id)_{g_f}\rd\mu_f+O(1)\\
=
-\frac{2d-2}{d^n}\int_{\sP^1}\Phi(f^n,\Id)_{g_f}\rd\mu_f+O(nd^{-n})\quad\text{as }n\to\infty,
\end{multline*}
which with $(0\neq)-(2-2d^{-n})/n+(2d-2)/d^n=O(n^{-1})$ as $n\to\infty$ yields
\begin{gather}
\int_{\sP^1}\Phi(f^n,\Id)_{g_f}\rd\mu_f=n\cdot O(1)=O(n)\quad\text{as }n\to\infty.\label{eq:dynamicallinear}
\end{gather}

\begin{proof}[Proof of Theorem $\ref{th:lyapunov}$]
 Once \eqref{eq:simple} and \eqref{eq:dynamicallinear} are at our disposal,
 we have
 \begin{multline*}
 \frac{1}{d^n}\int_{\sP^1\setminus\SAT(f)}\log(f^\#)\rd[f^n=\Id]
 -\frac{[f^n=\Id](\sP^1\setminus\SAT(f))}{d^n}L(f)\\
=O(nd^{-n}) \quad\text{as }n\to\infty,
 \end{multline*}
 the left hand side of which is computed as
 \begin{multline*}
 \frac{1}{d^n}\int_{\sP^1\setminus\SAT(f)}\log(f^\#)\rd[f^n=\Id]
 -\frac{[f^n=\Id](\sP^1\setminus\SAT(f))}{d^n}L(f)\\
 =\frac{1}{nd^n}\sum_{w\in\Fix(f^n)\setminus\SAT(f)}\log|(f^n)'(w)|-L(f)+O(d^{-n})
 \quad\text{as }n\to\infty
 \end{multline*}
 by the chain rule, $[f^n=\Id](\sP^1)=d^n+1$,
 and $\sup_{n\in\bN}([f^n=\Id](\SAT(f)))<\infty$. 
 Now the proof of $\eqref{eq:quantitativeapprox}$ is complete.

Let us show \eqref{eq:exact} using $\eqref{eq:quantitativeapprox}$:
for the details on the {\itshape M\"obius function} $\mu:\bN\to\{-1,0,1\}$,
which satisfies $\mu(1)=1$ by the definition, and
the {\itshape M\"obius inversion formula} below, see, e.g, \cite[\S 2]{Apostol}. 

For every $n\in\bN$, using the chain rule, we have
\begin{gather*}
 \frac{1}{n}\sum_{w\in\Fix(f^n)\setminus\SAT(f)}\log|(f^n)'(w)|
=\sum_{m\in\bN:\, m|n}
\frac{1}{m}\sum_{w\in\Fix^*(f^m)\setminus\SAT(f)}\log|(f^m)'(w)|,
\end{gather*}
which is equivalent to
\begin{multline}
\frac{1}{n}\sum_{w\in\Fix^*(f^n)\setminus\SAT(f)}\log|(f^n)'(w)|\\
=\sum_{m\in\bN:\, m|n}\mu\left(\frac{n}{m}\right)\cdot\frac{1}{m}\sum_{w\in\Fix(f^m)\setminus\SAT(f)}\log|(f^m)'(w)|\label{eq:mobius}
\end{multline}  
by the M\"obius inversion formula.

By $\#\SAT(f)<\infty$, for every $n\in\bN$ large enough,
we have $\Fix^*(f^n)\setminus\SAT(f)=\Fix^*(f^n)$. Hence
by \eqref{eq:mobius}, we have
\begin{multline*}
\left|\frac{1}{nd^n}\sum_{w\in\Fix^*(f^n)}\log|(f^n)'(w)|
-\frac{1}{nd^n}\sum_{w\in\Fix(f^n)\setminus\SAT(f)}\log|(f^n)'(w)|\right|\\
\le\frac{1}{d^n}\sum_{m\in\bN:\, m|n\text{ and }m<n}d^m
\left(\frac{1}{md^m}\sum_{w\in\Fix(f^m)\setminus\SAT(f)}\log|(f^m)'(w)|\right)\\
\le O(1)\cdot\frac{1}{d^n}\sum_{m=1}^{n/2}d^m
=O(d^{-n/2})\quad\text{as }n\to\infty,
\end{multline*}
where the second inequality is by
\eqref{eq:quantitativeapprox} and $\sup\{m\in\bN:m|n\text{ and }m<n\}\le n/2$.
Now the proof
of \eqref{eq:exact} is complete by \eqref{eq:quantitativeapprox}.
\end{proof}

\begin{remark}
The order estimate
\begin{gather*}
\frac{1}{md^m}\sum_{w\in\Fix(f^m)\setminus\SAT(f)}\log|(f^m)'(w)|=O(1)\quad\text{as }m\to\infty
\end{gather*}
is immediate if $f$ has at most
finitely many attracting periodic points in $\bP^1$.
\end{remark} 

\section{Proof of Theorem \ref{th:repelling}}
\label{sec:complex}

Suppose in addition that
$f$ has at most finitely many attracting
periodic points in $\bP^1$, or equivalently,
that there exists $N\in\bN$ such that for every $n\ge N$ and
every $w\in\Fix^*(f^n)$, we have $|(f^n)'(w)|\ge 1$.
Hence \eqref{eq:exact} implies \eqref{eq:strict}. 

For every $n\ge N$,
\begin{align*}
 &\frac{1}{nd^n}\sum_{w\in\Fix(f^n)\setminus\SAT(f)}\log|(f^n)'(w)|\\
 =&\frac{1}{nd^n}\sum_{m\in\bN:\, m|n\text{ and }m\ge N}\left(\sum_{w\in R_m^*(f)}\log|(f^n)'(w)|\right)+O((nd^n)^{-1})\\
=&\frac{1}{nd^n}\sum_{m\in\bN:\, m|n}\left(\sum_{w\in R_m^*(f)}\log|(f^n)'(w)|\right)+O((nd^n)^{-1})\\
 =&\frac{1}{nd^n}\sum_{w\in R(f^n)}\log|(f^n)'(w)|+O((nd^n)^{-1})
 \quad\text{as }n\to\infty,
 \end{align*}
so \eqref{eq:quantitativeapprox} implies \eqref{eq:repelling}. \qed

\begin{acknowledgement}
 The author thanks the referee for a very careful scrutiny and
 invaluable comments.
 This research was partially supported by JSPS Grant-in-Aid for Young Scientists (B), 24740087.
\end{acknowledgement} 


\def\cprime{$'$}

\end{document}